\documentclass[manuscript]{aomart}
	\usepackage{graphics}
	
	\newtheorem[{}\it]{thm}{Theorem}[section]
	\newtheorem{cor}[thm]{Corollary}
	\newtheorem{lem}[thm]{Lemma}

	\theoremstyle{definition}
	\newtheorem{defn}{Definition}[section]
	
	\newtheorem*[{}\it]{notation}{Notation}

	\newcommand{\eval}[2][\right]{\relax
		\ifx#1\right\relax \left.\fi#2#1\rvert}
	
	\title{The Distribution of the Nontrivial Zeros of Riemann Zeta Function}
	\author{Jianyun Zhang}
	\email{zhangjy20@outlook.com} 
	\copyrightnote{\textcopyright 2020, Jianyun Zhang, Beijing, P R China}
	\received{\formatdate{2020-7-17}}
	
\begin{document}
		
	\begin{abstract}
		We improve the estimation of the distribution of the nontrivial zeros of Riemann zeta function $\zeta(\sigma+it)$ for sufficiently large $t$, which is based on an exact calculation of some special logarithmic integrals of nonvanishing $\zeta(\sigma+it)$ along well-chosen contours. A special and single-valued coordinate transformation $s=\tau(z)$ is chosen as the inverse of $z=\chi(s)$, and the functional equation $\zeta(s) = \chi(s)\zeta(1-s)$ is simplified as $G(z) = z\, G_-(\frac{1}{z})$ in the $z$ coordinate, where $G(z)=\zeta(s)=\zeta\circ\tau(z)$ and $G_-$ is the conjugated branch of $G$. Two types of special and symmetric contours $\partial D_{\epsilon}^1$ and $\partial D_{\epsilon}^2$ in the $s$ coordinate are specified, and improper logarithmic integrals of nonvanishing $\zeta(s)$ along $\partial D_{\epsilon}^1$ and $\partial D_{\epsilon}^2$ can be calculated as $2\pi i$ and $0$ respectively, depending on the total increase in the argument of $z=\chi(s)$. Any domains in the critical strip for sufficiently large $t$ can be covered by the domains $D_{\epsilon}^1$ or $D_{\epsilon}^2$, and the distribution of nontrivial zeros of $\zeta(s)$ is revealed in the end, which is more subtle than Riemann's initial hypothesis and in rhythm with the argument of $\chi(\frac{1}{2}+it)$.
	\end{abstract}
	
	\maketitle
	
	\thispagestyle{plain}	
	\pagestyle{plain}
	
	\section{Introduction}
		The Riemann zeta function $\zeta(s)$ is one of the most challenging functions. Numerous arithmetic and analytic resorts of approximation, including integral transformations, have been exhausted with rare results on the zeros of $\zeta(s)$ due to its complicated definition. 
		
		Estimation of the number of zeros of $\zeta(\sigma+it)$ in the domain $\{\sigma+it|0<\sigma<1,0<t\le T\}$ was initially proposed by Riemann as $\frac{T}{2\pi}\log\frac{T}{2\pi}-\frac{T}{2\pi}+O(\log T)$, and was first verified by von Mangoldt \cite{Mangoldt05}. By the argument principle, the logarithmic integral of nonvanishing $\zeta(\sigma+it)$ along the boundary of a rectangle $[-1,2]\times[0,T]$ was estimated via Stirling's approximation and Jenson's formula \cite{Broughan17}. To improve the estimation of the distribution of zeros, we propose an exact calculation of some special logarithmic integrals of nonvanishing $\zeta(\sigma+it)$ along well-chosen contours for sufficiently large $t$ in the following three steps. 
		
		In the first step, a special and single-valued coordinate transformation $s=\tau(z)$ is chosen to simplify the functional equation \cite{Titchmarsh86} 
		\begin{equation}\label{basicfun}
		\zeta(s) = \chi(s)\zeta(1-s),\quad s\in \mathbb{C}\backslash\{1\}
		\end{equation}
		where the chi function
		\begin{equation}
		\chi(s) = 2^s\pi^{s-1}\sin(\frac{\pi{s}}{2})\Gamma(1-s)
		\end{equation}
		satisfies 
		\begin{equation}\label{chi0}
		\chi(s)\chi(1-s)=1.
		\end{equation} 
		
		The tau function $s=\tau(z)$ is chosen as the inverse of $z=\chi(s)$ with $s$ restricted in almost horizontal strips. The properties of $z=\chi(s)$ and its inverse $s=\tau(z)$ are discussed in detail in Section \ref{china} and Section \ref{argpresrv}. The chi function $z=\chi(s)$ is shown in Section \ref{china} to be multivalued when $t$ is sufficiently large, with argument monotone decreasing with respect to $t$. And in Section \ref{argpresrv}, we prove that the function $\tau(z)$, the inverse of $\chi(s)$, can be well defined based on the construction of a Riemann surface $R$. Further we introduce the concept of argument-preserving arcs $\gamma_{\phi}$, and prove that the function $\tau(z)$ is branched and its $m$-th branch maps a slit complex plane $S_m$ to an almost horizontal strip $D_m$ conformally, much similiar to the $m$-th branch of logarithm function $\log(z)$.
		
		In the second step, the functional equation \eqref{basicfun} is simplified as
		\begin{equation}
		G(z) = z\, G_-(\frac{1}{z})
		\end{equation}
		where 
		$G(z)=\zeta(s)=\zeta\circ\tau(z)$ and $G_-$ is the conjugated branch of $G$ defined at the beginning of Section \ref{tauContour}.
		
		This new functional equation relates a pair of conjugated branches $G$ and $G_-$, and claims that the logarithmic integrals of the two nonvanishing branches along a contour in the $z$ coordinate differ by $2\pi i$ or $0$, depending on the total increase in $\arg(z)$. 
		
		In the third step, two types of special and symmetric contours $\partial D_{\epsilon}^1$ and $\partial D_{\epsilon}^2$ in the $s$ coordinate are chosen to calculate some improper logarithmic integrals of nonvanishing $\zeta(s)$. 
		
		There are three major features of the chosen contours $\partial D_{\epsilon}^1$ and $\partial D_{\epsilon}^2$ defined in Section \ref{tauContour} to facilitate the calculation of integral. Firstly, instead of the total scale of $0\le t\le T$, the contours of $s$ are restricted within one horizontal strip $D_m$, such that the increase in the argument of $z=\chi(s)$ is no larger than $2\pi$, and the contours of $z$ stay in one slit complex plane $S_m$, without cutting through any branch cut. Secondly, the contours of $s$ meet the critical line at points $s_m=\frac{1}{2}+it+i\epsilon$ where $\zeta(s_m)$ is nonvanishing and $\chi(s_m)\to1$ as $\epsilon\to0$, such that $\zeta(s_m)\to\zeta(\overline{s_m})$ and $\log\zeta(s_m)\to\log\zeta(\overline{s_m})$ as $\epsilon\to0$. Thirdly, the contours of $s$ can be rather arbitrary such that any domains in the critical strip can be covered by the domains $D_{\epsilon}^1$ or $D_{\epsilon}^2$.
		
		At the end of Section \ref{tauContour}, we prove that
		\begin{equation}
		\lim\limits_{\epsilon\to0}\int_{\partial D_{\epsilon}^1}\frac{\zeta'(s)}{\zeta(s)}ds = 2\pi i,\quad\lim\limits_{\epsilon\to0}\int_{\partial D_{\epsilon}^2}\frac{\zeta'(s)}{\zeta(s)}ds = 0
		\end{equation}				
		with the main symbols illustrated in Figure \ref{myfigure}. By the argument principle we claim in Theorem \ref{RH} that all the zeros of $\zeta(\sigma+it)$ are on the critical line $\{\sigma+it|\sigma=\frac{1}{2}\}$ for sufficiently large $t$, and that there exists one and only one nontrivial zero of $\zeta(\frac{1}{2}+it)$ where $-2\pi (m+1)<\arg(\chi(\frac{1}{2}+it))<-2\pi m$ for any sufficiently large integer $m$. Roughly speaking, there exists one and only one nontrivial zero when $t-t\log\frac{t}{2\pi}$ decreases by $2\pi$ for sufficiently large $t$, which is a more subtle distribution than Riemann's initial hypothesis.

	\section{The modulus and argument of the chi function}
	\label{china}
		Preliminary properties of the special function $\chi(s)$ 
		are proposed in this section. Briefly speaking, the modulus and argument of $\chi(\sigma+it)$ are monotone decreasing about $\sigma$ and $t$ respectively, when $s$ is far away from the real axis in the complex plane. We only discuss the $s$ in the upper half-plane for the symmetry of $t$. 
		
		\begin{lem}\label{chi1}
			Let $s=\sigma+it$.
			Then $|\chi(s)|=1$ for $\sigma=\frac{1}{2}$. There also exists a real number $M_{1}>0$, such that the modulus of $\chi(s)$ is a continuous function of $\sigma$ and $t$ when $t\ge M_{1}$, satisfying the following properties:
			\begin{enumerate}
				\renewcommand{\labelenumi}{(\roman{enumi})}
				\item\label{rad:mono} $|\chi(s)|$ decreases strictly monotonously with increasing $\sigma$, and tends to $0$ as $\sigma\to+\infty$.
				\item\label{rad:1}  $0<|\chi(s)|<1$ for $\frac{1}{2}<\sigma<+\infty$. 
				\item\label{rad:2}  $1<|\chi(s)|<+\infty$ for $-\infty<\sigma<\frac{1}{2}$.		
			\end{enumerate}		 
		\end{lem}
		
		\begin{proof}
			Taking $\sigma=\frac{1}{2}$ in \eqref{chi0}, then $|\chi(s)|=1$. 
			
			Considering the zeros of $\sin(\frac{\pi{s}}{2})$ and the poles of $\Gamma(1-s)$, the chi function $\chi(s)$ is meromophic on the entire complex plane, with poles at $s=1,3,5,\dots$ and zeros at $s=0,-2,-4,\dots$ All the poles and zeros of $\chi(s)$ are on the real axis, then $0<|\chi(s)|<+\infty$ when $t\ne0$.
			
			The following asymptotic expansion \cite{Beals10} is adopted for real $x$ and $y$:
			\begin{equation}\label{absgamma}
			|\Gamma(x+iy)|=\sqrt{2\pi}|y|^{x-\frac{1}{2}}e^{-\frac{1}{2}\pi |y|}\big\{1+O(\frac{1}{|y|})\big\}, \quad |y|\rightarrow\infty.	
			\end{equation}	
			Taking $x=1-\sigma$ and $y=-t$ in \eqref{absgamma}, then 
			\begin{equation}\label{abschi}
			\begin{split}
			|\chi(s)| &= |2^{\sigma+it}|\cdot|\pi^{\sigma-1+it}|\cdot|\sin\big\{\frac{\pi}{2}(\sigma+it)\big\}|\cdot|\Gamma(1-\sigma-it)| \\
			&= 2^{\sigma}\pi^{\sigma-1}\sqrt{\sin^2(\frac{\pi}{2}\sigma)+\sinh^2(\frac{\pi}{2}t)}\sqrt{2\pi}t^{\frac{1}{2}-\sigma}e^{-\frac{1}{2}\pi t}\big\{1+O(\frac{1}{t})\big\}\\
			&= \sqrt{2\pi}2^{\sigma}\pi^{\sigma-1}\sqrt{\frac{\sin^2(\frac{\pi}{2}\sigma)+\sinh^2(\frac{\pi}{2}t)}{e^{\pi t}}}t^{\frac{1}{2}-\sigma}\big\{1+O(\frac{1}{t})\big\}.
			\end{split}
			\end{equation}
			The modulus $|\chi(s)|$ is a continuous function of $\sigma$ and $t$ when $t$ is sufficiently large. It's obvious that $0<|\chi(s)|<1$ for all $\sigma>\frac{1}{2}$, and $|\chi(s)|>1$ for all $\sigma<\frac{1}{2}$, both as $t\to+\infty$. 
			
			For any sufficiently large $t$, we also have $|\chi(s)|\to0$ as $\sigma\to+\infty$, and $|\chi(s)|\to+\infty$ as $\sigma\to-\infty$. 	
			
			The infinitesimal $\Delta|\chi(s)|$ with respect to $\Delta\sigma$ is 
			\begin{equation}
			\begin{split}
			\frac{\Delta|\chi(s)|}{|\chi(s)|} &=
			\big\{\log (2\pi)-\log t+\frac{\frac{\pi}{4}\sin\pi\sigma}{\sin^2(\frac{\pi}{2}\sigma)+\sinh^2(\frac{\pi}{2}t)}\big\}\big\{1+O(\frac{1}{t})\big\}\Delta\sigma
			\end{split}.
			\end{equation}
			It's obvious that $|\chi(s)|$ decreases strictly monotonously with increasing $\sigma$ when $t$ is sufficiently large. 
			
		\end{proof}
		Lemma \ref{chi1} demonstrates that the critical line $\{s|\textup{Re}(s)=\frac{1}{2}\}$ in the $s$-plane is mapped to a unit circle $S^1$ by $\chi(s)$. 
		
		\begin{lem}\label{chi2}
			Let $s=\sigma+it$.
			There exists a real number $M_{2}>0$, such that the argument of $\chi(s)$ is a continuous function of $\sigma$ and $t$ when $t\ge M_{2}$, satisfying the following properties:
			\begin{enumerate}
				\renewcommand{\labelenumi}{(\roman{enumi})}
				\item\label{ang:1} 
				$\arg(\chi(s))$ decreases strictly monotonously with increasing $t$, and tends to $-\infty$ as $t\to+\infty$. 
				\item\label{ang:2} $\arg(\chi(s))$ remains almost constant with varying $\sigma$. 
				\item\label{ang:3} $\arg(\chi(s))=\arg(\chi(1-\overline{s}))$. 
			\end{enumerate}	
		\end{lem}
		
		\begin{proof}		
			When $t\ne0$, the logarithm of $\chi(s)$ is obtained as
			\begin{equation}\label{logchi}
			\begin{split}
			\log\chi(\sigma+it) &= (\sigma+it)\log 2+(\sigma-1+it)\log\pi \\ 
			&\quad +\log\sin\big\{\frac{\pi}{2}(\sigma+it)\big\}+\log\Gamma(1-\sigma-it). 
			\end{split}
			\end{equation}
			The third item in \eqref{logchi} is 
			\begin{equation*}
			\begin{split}
			\log\sin\big\{\frac{\pi}{2}(\sigma+it)\big\} 
			&= \log\frac{e^{-\frac{\pi}{2}t}e^{i\frac{\pi}{2}\sigma}-e^{\frac{\pi}{2}t}e^{-i\frac{\pi}{2}\sigma}}{2i} \\
			&= \log\frac{-e^{\frac{\pi}{2}t}e^{-i\frac{\pi}{2}\sigma}}{2i}+\log(1-\frac{e^{i\pi\sigma}}{e^{\pi t}}) \\
			&= \frac{\pi}{2}t - \log2 + i\frac{\pi}{2}(1-\sigma) + O(-\frac{e^{i\pi\sigma}}{e^{\pi t}}).
			\end{split}
			\end{equation*}	
			The last item in \eqref{logchi} can be expressed by the asymptotic expansion \cite{Edelyi53}
			\begin{equation}\label{loggamma}
			\log\Gamma(z+a)=(z+a-\frac{1}{2})\log z-z+\frac{1}{2}\log(2\pi)+O(\frac{1}{z})
			\end{equation}
			where $|\arg z|<\pi$ and $a\in\mathbb{C}$. 
			Taking $z=-it$ and $a=\sigma$ in \eqref{loggamma}, then
			\begin{equation*}
			\log\Gamma(\sigma-it)= (-it+\frac{1}{2}-\sigma)\log (-it) + it +\frac{1}{2}\log(2\pi)+O(i\frac{1}{t}).
			\end{equation*}
			Finally 
			\begin{equation}\label{argchi}
			\begin{split}
			\arg(\chi(\sigma+it)) &= \textup{Im}(\log\chi(\sigma+it)) \\
			&= t-t\log\frac{t}{2\pi} + \frac{\pi}{4} + O(\frac{1}{t}) + O(-\frac{\sin\pi\sigma}{e^{\pi t}}). 
			\end{split}
			\end{equation}
			The argument $\arg(\chi(\sigma+it))$, if not restricted in the range of $2\pi$, is a continuous function of $\sigma$ and $t$ when $t$ is sufficiently large.
			It's obvious that 
			\begin{equation}\label{limargchi}
			\arg(\chi(\sigma+it))\to t-t\log\frac{t}{2\pi}+\frac{\pi}{4}, \quad t\to+\infty. 
			\end{equation}
			The $\arg(\chi(\sigma+it))$ decreases strictly monotonously with increasing $t$ when $t$ is sufficiently large, and $\arg(\chi(\sigma+it))$ tends to $-\infty$ as $t\to+\infty$. 
			
			It's also observed that $\arg(\chi(\sigma+it))$ remains almost constant with varying $\sigma$ when $t$ is sufficiently large.
			
			Since $\chi(s)\ne 0$ when $t\ne0$, the argument of \eqref{chi0} is
			\begin{equation}
			\begin{split}
			2\pi m &=\arg(\chi(\sigma+it))+\arg(\chi(1-\sigma-it))\\
			&=\arg(\chi(\sigma+it))-\arg(\chi(1-\sigma+it)), \quad m=0,\pm 1,\pm 2,\dots. 
			\end{split}
			\end{equation}
			And \eqref{limargchi} excludes the possibility of any non-zero $m$ for sufficiently large $t$, although $\arg(\chi(s))$ is multivalued. 
		\end{proof}
		Lemma \ref{chi2} demonstrates that for any real constant $\sigma$, while $s=\sigma+it$ moves upwards on the vertical line $\{s|\textup{Re}(s)=\sigma\}$, the image $\chi(s)$ wraps around the origin clockwise for $t\ge M_{2}$. Combined with Lemma \ref{chi1}, while $s$ moves upwards on the critical line $\{s|\textup{Re}(s)=\frac{1}{2}\}$, the image $\chi(s)$ loops on the unit circle $S^1$ clockwise, passing through one point such as $\{1\}$ on $S^1$ infinitely many times.

	\section{The branch and the inverse of the chi function}	
	\label{argpresrv}
	    The function $\tau(z)$, the inverse of $\chi(s)$, is defined in this section based on the construction of a Riemann surface and the concept of argument-preserving arcs. Briefly speaking, when Lemma \ref{chi1} and Lemma \ref{chi2} apply, the inverse of $\chi(s)$ can be well defined as tau function $s=\tau(z)$ in \eqref{tau} which is branched, and each branch maps a slit complex plane to an almost horizontal strip conformally, as proved in Lemma \ref{conformalMap}.
	    
	    Lemma \ref{chi1} and Lemma \ref{chi2} apply when $t\ge M_1$ and $t\ge M_2$ respectively, and the following domain is focused on: 
	    \begin{defn}\label{FAR}
	    	The domain $D\subset\mathbb{C}$ 
	    	is said to be far away from the real axis(FAR for short), if a proper real number $M\ge \max(M_{1},M_{2})$ can be chosen, such that $\arg(\chi(s))\le \arg(\chi(\frac{1}{2}+iM))$ for all $s\in D$.
	    \end{defn}
	    As a special case, the veritical line $\{s|\sigma=\frac{1}{2},t\ge M\}$ is FAR for any $M\ge \max(M_{1},M_{2})$, since $\arg(\chi(\frac{1}{2}+it))\le \arg(\chi(\frac{1}{2}+iM))$ for all the points on the half line by Lemma \ref{chi2}. Definition \ref{FAR} can be extended to the lower half-plane as $\arg(\chi(\sigma+it))\ge \arg(\chi(\frac{1}{2}+iM))$ for $t<0$.
	    
	    For all $s\in D$, we observe that $0<|\chi(s)|<+\infty$ by Lemma \ref{chi1}, and that $\arg(\chi(s))$ is monotonous decreasing about $t$ and tends to $-\infty$ by Lemma \ref{chi2}. Hence a Riemann surface can be constructed as follows to make the range of $z=\chi(s)$ single-valued.
	    
	    The slit $z$-plane $\mathbb{C}\backslash[0,+\infty)$ is designated as the $m$th sheet $S_{m}$, where $m$ is any integer sufficiently large. And when every two sheets $S_{m}$ and $S_{m+1}$ are attached along the branch cut $(0,+\infty)$, a Riemann surface $R$ spread over the $z$-plane is constructed.
	    
	    Then the chi function 
	    \begin{equation}\label{Rchi}
	    z=\chi(s),\quad s\in D 
	    \end{equation}
	    where $z\in R$, is the composition of infinitely many branches
	    \begin{equation}\label{multichi}
	    z=\chi_{m}(s)=|\chi(s)|e^{\textup{Arg}(\chi(s))-2\pi im} 
	    \end{equation} 
	    where $z\in S_{m}\subset R$. 
	    
	    \begin{defn}
	    	Let $\phi\in \mathbb{R}$ be a constant. The arc $\gamma_{\phi}$ in the $s$-plane is said to be argument-preserving for the map $\chi:s\to z$, if $\arg(z)=\phi$ for all $s\in \gamma_{\phi}$.
	    \end{defn}
	    
	    \begin{lem}\label{argpreserving}
	    	Let $D=\{s|\sigma_1\le\textup{Re}(s)\le\sigma_2\}$ be a FAR domain where $\sigma_1<\frac{1}{2}<\sigma_2$.
	    	For any $s_{0}\in D$, there exists a unique argument-preserving arc $\gamma_{\phi}$ for the map $\chi$, such that $s_{0}\in\gamma_{\phi}$ and the arc $\gamma_{\phi}$ splits the domain $D$ 
	    	horizontally into two parts.
	    \end{lem}
	    \begin{proof}
	    	Let $\chi(\sigma+it)=r(\sigma,t)e^{i\phi(\sigma,t)}$ where $\phi(\sigma,t)=\arg(\chi(\sigma+it))$ and $r(\sigma,t)=|\chi(\sigma+it))|$.
	    	
	    	Step 1. Choose $M$ for the domain $D$.
	    	
	    	From \eqref{argchi} we know that $\phi(\sigma,t)$ is continuous and bounded for any finite $\sigma$ and $t$. 
	    	Let 
	    	$$\epsilon=\sup\limits_{\sigma\in[\sigma_1,\sigma_2]}|\phi(\sigma,t)-\phi(\frac{1}{2},t)|.$$
	    	There exists a real number $M_{3}>0$ and an arbitrarily small $\epsilon_{0}<\pi$, such that $\epsilon<\epsilon_{0}$ for all $t\ge M_{3}$, because $\epsilon\to 0$ as $t\to +\infty$ by Lemma \ref{chi2}.
	    	
	    	Let $M_{0}=\max(M_{1},M_{2},M_{3})$. 
	    	Then $|\phi(\sigma,M_{0})-\phi(\frac{1}{2},M_{0})|<\epsilon_{0}$ for all $\sigma\in[\sigma_1,\sigma_2]$. 
	    	Take a real number $M>M_{0}$ satisfying $\phi(\frac{1}{2},M)=\phi(\frac{1}{2},M_{0})-2\epsilon_{0}$. Then for all $\sigma\in[\sigma_1,\sigma_2]$ and $t\ge M$,
	    	\begin{equation}\label{inequality}
	    	\phi(\sigma,M)<\phi(\frac{1}{2},M)+\epsilon_{0}=\phi(\frac{1}{2},M_{0})-\epsilon_{0}<\phi(\sigma,M_{0}).
	    	\end{equation}  
	    	
	    	Choose the domain $D=\{s|\sigma_1\le\sigma\le\sigma_2,t\ge M\}$ as the FAR critical strip.
	    	
	    	Step 2. The existance of the unique set $\gamma_{\phi}$.
	    	
	    	If $s_{0}=\sigma_{0}+it_{0}\in D$, then $t_{0}\ge M>M_{0}$.
	    	For any constant $\sigma\in[\sigma_1,\sigma_2]$, a continuous real function of $t$ is constructed as
	    	\begin{equation}
	    	F(t)=\phi(\sigma,t)-\phi(\sigma_{0},t_{0}).
	    	\end{equation}	
	    	When $t=M_{0}$, we obtain that $F(M_{0})=\phi(\sigma,M_{0})-\phi(\sigma_{0},t_{0})>0$ by \eqref{inequality}. 
	    	When $t\to\infty$, we obtain that $F(t)=\phi(\sigma,t)-\phi(\sigma_{0},t_{0})$ decreases strictly monotonously and tends to $-\infty$ by Lemma \ref{chi2}.
	    	
	    	Therefore there exists a unique and bounded $\hat{t}$ satisfying $F(t)=0$ for each $\hat{\sigma}\in[\sigma_1,\sigma_2]$, 
	    	all of which form the unique set $\gamma_{\phi}=\{\sigma+it|\phi(\sigma,t)=\phi(\sigma_{0},t_{0}),\sigma_1\le\sigma\le\sigma_2\}$.
	    	
	    	Step 3. The set $\gamma_{\phi}$ is a continuous and simple arc.
	    	
	    	The two-variable function $H(\sigma,t)=\phi(\sigma,t)-\phi(\sigma_{0},t_{0})$ is continuous about $\sigma$ and $t$ for $\sigma\in(\sigma_1-\delta,\sigma_2+\delta)$ and $t\ge M$, where $\delta>0$ is small. Because $H(\sigma,t)$ is monotone decreasing about $t$, a unique implicit function $t=h(\sigma)$ can be established from $H(\sigma,t)=0$ near any of its solution $(\hat{\sigma},\hat{t})$ which is bounded. And $t=h(\sigma)$ is continuous and bounded near any $\hat{\sigma}\in[\sigma_1,\sigma_2]$.
	    	
	    	Therefore the single-valued function $t=h(\sigma)$ is continuous and bounded for all $\sigma\in[\sigma_1,\sigma_2]$, which shows that $\gamma_{\phi}$ is a continuous and simple arc, splitting the critical strip $\{\sigma+it|\sigma_1\le\sigma\le\sigma_2\}$ horizontally into two parts.
	    \end{proof} 
	    
	    The argument-preserving arc $\gamma_{\phi}\subset D$ is mapped to $\beta_{\phi}\subset R$ by the map  
	    \begin{equation}
	    \chi:\gamma_{\phi}\to\beta_{\phi},
	    \end{equation}
	    and $\beta_{\phi}$ is a line segment on a ray issuing from the origin.   
	    
	    \begin{cor}\label{chiarc}	
	    	The	map $\chi:\gamma_{\phi}\to \beta_{\phi}$ is continuous and one-to-one,
	    	satisfying the following properties:
	    	\begin{enumerate}
	    		\renewcommand{\labelenumi}{(\roman{enumi})}
	    		\item\label{map:1} If $s_{0}\in \gamma_{\phi}$, then $(1-\overline{s_{0}})\in \gamma_{\phi}$. 
	    		\item\label{map:2} Two arcs $\gamma_{\phi_{1}}$ and $\gamma_{\phi_{2}}$ do not intersect in the $s$-plane, if $\phi_{1}\ne\phi_{2}$.
	    	\end{enumerate}	
	    \end{cor}
	    
	    \begin{proof}
	    	For any $s\in\gamma_{\phi}$, let $\chi(s)=r(\sigma,t)e^{i\phi}$ where $\phi$ is a constant real.
	    	By Lemma \ref{chi1} the radius $r$ is continuous and monotone decreasing with respect to $\sigma$. Then the map $\chi:\gamma_{\phi}\to \beta_{\phi}$ is continuous and one-to-one.
	    	
	    	Lemma \ref{chi2} claims that $\arg(\chi(s_{0}))=\arg(\chi(1-\overline{s_{0}}))$, which makes property (i) true.
	    	
	    	Suppose $s_{0}$ is the point where $\gamma_{\phi_{1}}$ and $\gamma_{\phi_{2}}$ intersect. The assumption of  $\gamma_{\phi_{1}}\ne\gamma_{\phi_{2}}$ with $\phi_{1}=\phi_{2}$ contradicts Lemma \ref{argpreserving}, which makes property (ii) true.
	    \end{proof}
	    
	    The FAR domain $D$ is mapped to $U$ by the map  
	    \begin{equation}
	    \chi:D\to U
	    \end{equation}
	    where $U$ is the domain in the Riemann surface $R$. 
	    
	    When the branch cut $[0,+\infty)$ is chosen, the domain $U$ can be separated into infinitely many sheets as 
	    \begin{equation}
	    U_m=U(\phi_2,\phi_1)=\bigcup_{\phi_2<\phi< \phi_1}\beta_{\phi}
	    \end{equation}
	    where $m$ is any sufficiently large integer and $$\phi_1=-2\pi m,\quad \phi_2=\phi_1-2\pi.$$
	    Each $U_m$ is a domain in the slit complex plane $S_m$. Correspondingly the domain $D$ can be separated into infinitely many horizontal strips as  
	    \begin{equation}
	    D_m=D(\phi_2,\phi_1)=\bigcup_{\phi_2<\phi< \phi_1}\gamma_{\phi}.
	    \end{equation}
	    Each $D_m$ is the pre-image of $U_m$.
	    
	    \begin{cor}\label{chiPatch}
	    	The	map $\chi:D_m\to U_m$ is one-to-one, and the inverse map $\chi^{-1}:U_m\to D_m$ can be defined on the whole slit plane $S_m$. 
	    \end{cor}
	    \begin{proof}
	    	Suppose $s=\sigma+it\in D_m$ and $z=re^{i\phi}\in U_m$.
	    	
	    	If $\chi(s_1)=\chi(s_2)=r_0e^{i\phi_0}$ where $s_1,s_2\in D_m$, then $s_1,s_2\in\gamma_{\phi_0}$. Corollary \ref{chiarc} requires that $s_1=s_2$ for they share a single $r_0$. Then the map $\chi$ is one-to-one from $D_m$ to $U_m$. 
	    	
	    	Let $\chi:\gamma_{\phi}\to\beta_{\phi}$ where $\gamma_{\phi}$ is the arc preserving the argument $\phi$ for the map $\chi$, and let $$D_0=D_m\cap\{s|\sigma_1\le\textup{Re}(s)\le\sigma_2\}$$ where $\sigma_1<\frac{1}{2}<\sigma_2$.
	    	
	    	Lemma \ref{argpreserving} claims that $\gamma_{\phi}$ always intersects the vertical lines $\{s|\sigma=\sigma_1\}$ and $\{s|\sigma=\sigma_2\}$ in the FAR domain $D_0$, even as $\sigma_2\to+\infty$ and $\sigma_1\to-\infty$.		
	    	Lemma \ref{chi1} continues to claim that $\beta_{\phi}$ is a line segment with one end tending to the origin, and with the other end tending to $\infty$. As $\phi$ decreases by $2\pi$, the domain $U_m$ covers the whole complex plane $\mathbb{C}$ except the branch cut. 
	    	
	    \end{proof}
	    
	    The tau function $\tau(z)$, inverse of the $m$-th branch function \eqref{multichi}, can be well defined based on Corollary \ref{chiPatch} as 
	    \begin{equation}\label{tau}
	    s=\tau(z)=\chi^{-1}(z), \quad z\in \mathbb{C}\backslash [0,+\infty)
	    \end{equation} 
	    where $-2\pi (m+1)<\arg(z)<-2\pi m$. 
	    
	    \begin{lem}\label{conformalMap}	
	    	The tau function $s=\tau(z)$ is a conformal mapping of the slit plane $S_m$ onto the horizontal strip $D_m$. 
	    \end{lem}
	    \begin{proof}
	    	Suppose $\chi:s\to z$ where $s=\sigma+it\in D_m$ and $z=re^{i\phi}\in S_m$.
	    	
	    	The FAR domain $D_m$ contains neither zeros nor poles of $\chi(s)$. 
	    	Since $\chi(s)$ is analytic for all $s\in D_m$, 
	    	the function $\chi(s)$ is differentiable for all $s\in D_m$,
	    	and the derivative of the nonzero $\chi(s)=re^{i\phi}$ is
	    	\begin{equation}
	    	\begin{split}
	    	\chi'(s)&=\frac{\partial(r\cos\phi)}{\partial\sigma} + i\frac{\partial(r\sin\phi)}{\partial\sigma}\\
	    	&=\frac{\partial r}{\partial\sigma}e^{i\phi} + i\frac{\partial\phi}{\partial\sigma}\chi(s).
	    	\end{split}
	    	\end{equation}
	    	
	    	Both Lemma \ref{chi1} and Lemma \ref{chi2} applies for all $s\in D_m$, which require $\frac{\partial r}{\partial\sigma}<0$ and $\frac{\partial\phi}{\partial\sigma}\to0$. Therefore $\chi'(s)\ne0$ for all $s\in D_m$. And the inverse map $\tau:z\to s$ is analytic on $S_m$ by implicit function theorem.	
	    \end{proof}	 
	    
	    The tau function $\tau(z)$ can also be defined on the negative $m$-th sheet.
	    \begin{defn}\label{conjBranch}
	    	The branch 
	    	\begin{equation}\label{tau-}
	    	s=\tau_{-}(z), \quad z\in S_m^*
	    	\end{equation}
	    	is said to be the conjugated branch of 
	    	\begin{equation}\label{tau+}
	    	s=\tau(z), \quad z\in S_m
	    	\end{equation}
	    	if $S_m^{*}=\{\overline{z}|z\in S_m\}$ is the reflection of $S_m$.
	    \end{defn}
	    The conjugated branch of \eqref{tau} is
	    \begin{equation}\label{tau1-}
	    s=\tau_-(z)=\chi^{-1}(z), \quad z\in \mathbb{C}\backslash [0,+\infty)
	    \end{equation} 
	    where $2\pi m<\arg(z)< 2\pi (m+1)$.
	    
	    Another pair of conjugated branches are available if $(-\infty,0]$ is chosen to be the branch cut. The corresponding tau functions are 
	    \begin{equation}\label{tau2+}
	    s=\tau(z)=\chi^{-1}(z), \quad z\in \mathbb{C}\backslash (-\infty,0]
	    \end{equation} 
	    where $-\pi-2\pi m<\arg(z)<\pi-2\pi m$, and 
	    \begin{equation}\label{tau2-}
	    s=\tau_-(z)=\chi^{-1}(z), \quad z\in \mathbb{C}\backslash (-\infty,0]
	    \end{equation} 
	    where $-\pi + 2\pi m<\arg(z)< \pi + 2\pi m$. 
	    
	    Generally speaking, the topology of $\chi(s)$ is similiar to the topology of $e^s$, when $s$ is far away from the real axis in the complex plane. And the topology of $\tau(z)$ is similiar to the topology of $\log(z)$.

	\section{Special logarithmic integral of the Riemann zeta function}
	\label{tauContour}	
		Special logarithmic integrals of nonvanishing $\zeta(s)$ are calculated along two types of well-chosen contours in this section. Briefly speaking, functional equation \eqref{basicfun} for $\zeta(s)$ can be largely simplified as \eqref{funEquation} by the coordinate transformation $s=\tau(z)$, and the improper logarithmic integrals of nonvanishing $\zeta(s)$ can be calculated exactly if the contours in the $s$ coordinate satisfy several features in Definition \ref{simpleD1} and Definition \ref{simpleD2}. Riemann hypothesis is verified based on Lemma \ref{Res1} and Lemma \ref{Res2} since the contours can be chosen arbitrarily, and more detail about the distribution of nontrivial zeros of $\zeta(s)$ is revealed in Theorem \ref{RH}.
		
		The composite function is introduced based on \eqref{tau} or \eqref{tau2+} as
		\begin{equation}\label{multifun}
		w=\zeta(s)=\zeta\circ\chi^{-1}(z)=\zeta\circ\tau(z)=G(z),\quad z\in S_m.
		\end{equation}
		where $-2\pi (m+1)<\arg(z)<-2\pi m$ or $-\pi-2\pi m<\arg(z)<\pi-2\pi m$, and $m$ is sufficiently large. 
		
		The function $G(z)$ is also branched in accordance with the branch of $\tau(z)$. And $G(z)$ is analytic on the slit plane $S_m$ by Lemma \ref{conformalMap}.
		
		By Definition \ref{conjBranch} the conjugated branch of \eqref{multifun} is
		\begin{equation}\label{G-}
		w=\zeta(s)=\zeta\circ\tau_-(z)=G_-(z),\quad z\in S_m^*
		\end{equation}
		where $S_m^*$ is the reflection of $S_m$.
		
		Let the point $z\in S_m$ and 
		\begin{equation}\label{Point1}
		G:z\stackrel{\tau}{\to}s \stackrel{\zeta}{\to}\zeta(s).
		\end{equation}		
		Supposing $\eta=\chi(1-s)$, by \eqref{chi0} we have
		\begin{equation}
		\eta=\chi(1-s)=\frac{1}{\chi(s)}=\frac{1}{z}.
		\end{equation}
		Lemma \ref{chi2} claims that $\arg(\overline{\eta})=\arg(z)$ and $\arg(\eta)=-\arg(z)$, requiring the points $\overline{\eta}\in S_m$ and $\eta\in S_m^{*}$ respectively and
		\begin{equation}\label{Point23}
		\begin{split}
		G&:\overline{\eta}=\overline{(\frac{1}{z})}\stackrel{\tau}{\to}1-\overline{s} \stackrel{\zeta}{\to}\overline{\zeta(1-s)},\\
		G_-&:\eta=\frac{1}{z}\stackrel{\tau_-}{\to}1-s \stackrel{\zeta}{\to}\zeta(1-s).
		\end{split}		
		\end{equation}	
		
		The fundamental functional equation \eqref{basicfun} can be rewritten as
		\begin{equation}\label{funEquation}
		G(z) = z\, G_-(\frac{1}{z})
		\end{equation}		
		
		Let $\tau:S_m\to D_m$ and the domain $D_m$ is a horizontal strip  
		\begin{equation}\label{HStrip}
		D_m=D(\phi_2,\phi_1)=\bigcup_{\phi_2<\phi< \phi_1}\gamma_{\phi} 
		\end{equation}
		where $\phi_1=-2\pi m$ or $\phi_1=\pi-2\pi m$, and $\phi_2=\phi_1-2\pi$.
		
		We come to study the logarithmic integral of nonvanishing zeta function
		\begin{equation}\label{logintzeta}
		\int_{\partial D_{\epsilon}}\frac{\zeta'(s)}{\zeta(s)}ds
		\end{equation}
		where the domain $D_{\epsilon}$ is defined as follows.
		\begin{defn}\label{simpleD}
			A domain $D_{\epsilon}$ is said to be simple, if 
			\begin{enumerate}
				\renewcommand{\labelenumi}{(\roman{enumi})}
				\item  the domain $D_{\epsilon}\subset D_m$ is bounded, simple connected, and symmetric with respect to the vertical line $\{s|\textup{Re}(s)=\frac{1}{2}\}$;
				\item  the boundary $\partial D_{\epsilon}$ is piecewise smooth, and $\zeta(s)\ne0$ for all $s\in \partial D_{\epsilon}$;
				\item  the boundary $\partial D_{\epsilon}$ meets the vertical line $\{s|\textup{Re}(s)=\frac{1}{2}\}$ only twice.
			\end{enumerate}			
		\end{defn}
		
		Two types of simple domains $D_{\epsilon}$ are studied. 
		The illustration of type-one domain $D_{\epsilon}^1$ can be found in Figure \ref{myfigure}.
		\begin{defn}\label{simpleD1}
			A domain $D_{\epsilon}^1$ is said to be type-one, if 
			\begin{enumerate}
				\renewcommand{\labelenumi}{(\roman{enumi})}
				\item  the domain $D_{\epsilon}^1\subset D_m$ is simple and $$D_m=D(\phi_2,\phi_1)=\bigcup_{\phi_2<\phi< \phi_1}\gamma_{\phi}$$ where $\phi_1=-2\pi m$ and $\phi_2=-2\pi (m+1)$;
				\item  the boundary $\partial D_{\epsilon}^1$ meets the vertical line $\{s|\textup{Re}(s)=\frac{1}{2}\}$ at two points
				$$s_{m+1}^-=\frac{1}{2}+i(t_{m+1}-\epsilon) \quad\text{and}\quad s_m^+=\frac{1}{2}+i(t_m+\epsilon)$$
				where $\arg(\chi(\frac{1}{2}+it_m))=-2\pi m$, $\arg(\chi(\frac{1}{2}+it_{m+1}))=-2\pi (m+1)$, and $\epsilon>0$ is arbitrarily small.
			\end{enumerate}			
		\end{defn}
		
		\begin{defn}\label{simpleD2}
			A domain $D_{\epsilon}^2$ is said to be type-two, if 
			\begin{enumerate}
				\renewcommand{\labelenumi}{(\roman{enumi})}
				\item  the domain $D_{\epsilon}^2\subset D_m$ is simple and $$D_m=D(\phi_2,\phi_1)=\bigcup_{\phi_2<\phi< \phi_1}\gamma_{\phi}$$ where $\phi_1=\pi-2\pi m$ and $\phi_2=-\pi-2\pi m$;
				\item  the boundary $\partial D_{\epsilon}^2$ meets the vertical line $\{s|\textup{Re}(s)=\frac{1}{2}\}$ at two points
				$$s_m^+=\frac{1}{2}+i(t_m+\epsilon) \quad\text{and}\quad s_m^-=\frac{1}{2}+i(t_m-\epsilon)$$
				where $\arg(\chi(\frac{1}{2}+it_m))=-2\pi m$, and $\epsilon>0$ is arbitrarily small.
			\end{enumerate}			
		\end{defn}
		
		\begin{lem}\label{Res1}
			The logarithmic integral of nonvanishing zeta function 
			\begin{equation}
			\lim\limits_{\epsilon\to0}\int_{\partial D_{\epsilon}^1}\frac{\zeta'(s)}{\zeta(s)}ds = 2\pi i
			\end{equation}
			where $D_{\epsilon}^1$ is type-one simple domain defined in Definition \ref{simpleD1}.			
		\end{lem}
		\begin{proof}
			The main symbols are illustrated in Figure \ref{myfigure}.
			
			Step 1. Choose $\epsilon$ for the domain $D_{\epsilon}^1$.
			
			When $s=\frac{1}{2}+it$, we have $1-s=\frac{1}{2}-it=\overline{s}$. By \eqref{basicfun} we obtain
			\begin{equation}
			\zeta(s)-\overline{\zeta(s)} = (\chi(s)-1)\overline{\zeta(s)},
			\end{equation}
			which requires $\zeta(s)$ to be real when $\chi(s)=1$. 
			
			Since the zeros of analytic $\zeta(s)$ are isolated, there exists a small $\epsilon_0>0$ such that $\zeta(\frac{1}{2}+i(t_m+\epsilon))\ne0$ and $\zeta(\frac{1}{2}+i(t_{m+1}+\epsilon))\ne0$ for all $0<|\epsilon|<\epsilon_0$.
			
			Choose $0<\epsilon<\epsilon_0$ for the domain $D_{\epsilon}^1$, where
			\begin{equation}\label{realValue}
			\begin{split}
			\lim\limits_{\epsilon\to0}\zeta(s_m^+)=\lim\limits_{\epsilon\to0}\overline{\zeta(s_m^+)}=\lim\limits_{\epsilon\to0}\zeta(\overline{s_m^+}), \\ \lim\limits_{\epsilon\to0}\zeta(s_{m+1}^-)=\lim\limits_{\epsilon\to0}\overline{\zeta(s_{m+1}^-)}=\lim\limits_{\epsilon\to0}\zeta(\overline{s_{m+1}^-}),
			\end{split}			
			\end{equation}
			since $\chi(s)\to1$ and $\zeta(s)\to\overline{\zeta(s)}$ as $\epsilon\to0$.
			
			Let $\gamma$ and $\overline{\gamma}$ be any arc and its conjugate respectively, where $\gamma$ starts at $s_{m+1}^-$ and ends at $s_m^+$. When $\zeta(s)\ne0$ on $\gamma$ or $\overline{\gamma}$, the lorgarithmic integral
			\begin{equation}\label{conjInt}
			\begin{split}
			\lim\limits_{\epsilon\to0}\Big\{\int_{\gamma}\frac{\zeta'(s)}{\zeta(s)}ds - \int_{\overline{\gamma}}\frac{\zeta'(s)}{\zeta(s)}ds\Big\} &=  \lim\limits_{\epsilon\to0}\Big\{\log\zeta(s)\Big|_{s_{m+1}^-}^{s_m^+} - \log\zeta(s)\Big|_{\overline{s_{m+1}^-}}^{\overline{s_m^+}}\Big\} \\
			&= 0	  			 
			\end{split}
			\end{equation}
			for $\epsilon\in(0,\epsilon_0)$.
			
			Step 2. Integrate along the boundary $\partial D_{\epsilon}^1$.
			
			Suppose the boundary $\partial D_{\epsilon}^1\subset D_m=D(\phi_2,\phi_1)$ is separated into two arcs $\gamma^0$ and $\gamma=\{1-\overline{s}|s\in \gamma^0\}$ by the vertical line $\{s|\textup{Re}(s)=\frac{1}{2}\}$. 
			And the arc $\overline{\gamma}=\{\overline{s}|s\in \gamma\}$ is the conjugate of the arc $\gamma$.
			
			Let $$\chi:\gamma^0\to\beta^0,\quad\gamma\to\beta,\quad\overline{\gamma}\to\overline{\beta}.$$ 
			Then the arc $\overline{\beta}=\{z|\frac{1}{z}\in\beta^0\}\subset S_m^*$ and the arc $\beta=\{z|\overline{z}\in\overline{\beta}\}\subset S_m$.
			And the map $\chi$ sends both $\gamma_{\phi_1}$ and $\gamma_{\phi_2}$ to the branch cut $(0,+\infty)$.
			
			When the point $s\in\gamma^0$ starts at $s_{m+1}^-$ and ends at $s_m^+$ in the counterclockwise direction of $\partial D_{\epsilon}^1$, the point $1-\overline{s}\in\gamma$ starts at $s_{m+1}^-$ and ends at $s_m^+$ in the clockwise direction of $\partial D_{\epsilon}^1$, and the point $1-s\in\overline{\gamma}$ starts at $\overline{s_{m+1}^-}$ and ends at $\overline{s_m^+}$.
			
			As $s$ describes the arc $\gamma^0$, the value $z=\chi(s)\in\beta^0$ moves continuously, starting at $1+i0^+$ on the top edge of the branch cut $(0,+\infty)$ and ending at $1+i0^-$ on the bottom edge of $(0,+\infty)$, with $\arg(z)$ increasing from $\phi_2=-2\pi (m+1)$ to $\phi_1=-2\pi m$.
			
			For nonvanishing $\zeta(s)$, taking the logarithm of \eqref{funEquation}, 
			we obtain
			\begin{equation}\label{funEquLog}
			\log G(z) = \log z + \log G_-(\frac{1}{z})
			\end{equation}
			The derivative of \eqref{funEquLog} with respect to $z$ is
			\begin{equation}\label{funEquDlog}
			\frac{G'(z)}{G(z)} = \frac{1}{z} - \frac{1}{z^2}\frac{G'_-(\frac{1}{z})}{G_-(\frac{1}{z})},		
			\end{equation}
			which can be integrated along the arc $\beta^0\subset S_m$ as
			\begin{equation}\label{GInt}
			\begin{split}
			\int_{\beta^0}\frac{G'(z)}{G(z)}dz &= \int_{\beta^0}\frac{dz}{z} - \int_{\beta^0}\frac{1}{z^2}\frac{G'_-(\frac{1}{z})}{G_-(\frac{1}{z})}dz \\
			&= \int_{1+i0^+}^{1+i0^-}\frac{dz}{z} + \int_{\overline{\beta}}\frac{G'_-(\eta)}{G_-(\eta)}d\eta \\
			&\to 2\pi i + \int_{\overline{\beta}}\frac{G'_-(\eta)}{G_-(\eta)}d\eta, \quad\epsilon\to0.
			\end{split}		
			\end{equation}
			
			For nonvanishing $\zeta(s)$, by \eqref{conjInt} and \eqref{GInt} we obtain
			\begin{equation}\label{TriSteps}
			\begin{split}
			\lim\limits_{\epsilon\to0}\int_{\partial D_{\epsilon}^1}\frac{\zeta'(s)}{\zeta(s)}ds &= 
			\lim\limits_{\epsilon\to0}\Big\{\int_{\gamma^0}\frac{\zeta'(s)}{\zeta(s)}ds - \int_{\gamma}\frac{\zeta'(s)}{\zeta(s)}ds\Big\} \\
			&= \lim\limits_{\epsilon\to0}\Big\{\int_{\gamma^0}\frac{\zeta'(s)}{\zeta(s)}ds - \int_{\overline{\gamma}}\frac{\zeta'(s)}{\zeta(s)}ds\Big\} \\
			&= \lim\limits_{\epsilon\to0}\Big\{\int_{\beta^0}\frac{G'(z)}{G(z)}dz - \int_{\overline{\beta}}\frac{G'_-(\eta)}{G_-(\eta)}d\eta\Big\}	\\
			&= 2\pi i.		
			\end{split}		
			\end{equation}
			
		\end{proof}		

		\begin{figure}			
			\vspace{5pt}
			\scalebox{0.5}[0.5]{\includegraphics{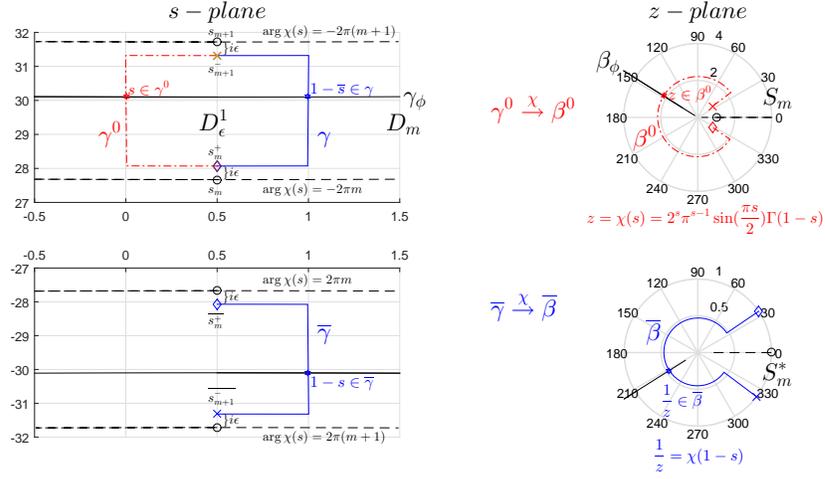}}
			\caption[]{Illustration of the proof of Lemma \ref{Res1}.}
			\label{myfigure}
		\end{figure}
		
		\begin{lem}\label{Res2}			
			The logarithmic integral of nonvanishing zeta function 
			\begin{equation}
			\lim\limits_{\epsilon\to0}\int_{\partial D_{\epsilon}^2}\frac{\zeta'(s)}{\zeta(s)}ds = 0
			\end{equation}
			where $D_{\epsilon}^2$ is type-two simple domain defined in Definition \ref{simpleD2}.			
		\end{lem}
		\begin{proof}
			
			We only sketch the proof here.
			
			The branch cut is chosen to be $(-\infty,0]$, such that all the images of the arcs broken from $\partial D_{\epsilon}^2$ or from its conjugate $\partial \overline{D_{\epsilon}^2}$ stay inside one sheet $S_m$ or inside its conjugate $S_m^*$.
			
			When the point $s\in\gamma^0$ starts at $s_m^+$ and ends at $s_m^-$ in the counterclockwise direction of $\partial D_{\epsilon}^2$, the point $1-\overline{s}\in\gamma$ starts at $s_m^+$ and ends at $s_m^-$ in the clockwise direction of $\partial D_{\epsilon}^2$, and the point $1-s\in\overline{\gamma}$ starts at $\overline{s_m^+}$ and ends at $\overline{s_m^-}$.
			
			As $s$ describes the arc $\gamma^0$, the value $z=\chi(s)\in\beta^0$ moves continuously, starting at $1+i0^+$ on the top edge of $(0,+\infty)$ and ending at $1+i0^-$ on the bottom edge of $(0,+\infty)$, without cutting through the branch cut $(-\infty,0]$, and the total increase in $\arg(z)$ is zero.
			
			That makes the difference in \eqref{GInt}, which leads to
			\begin{equation}
			\begin{split}
			\int_{\beta^0}\frac{G'(z)}{G(z)}dz &= \int_{\beta^0}\frac{dz}{z} - \int_{\beta^0}\frac{1}{z^2}\frac{G'_-(\frac{1}{z})}{G_-(\frac{1}{z})}dz \\
			&= \int_{1+i0^+}^{1+i0^-}\frac{dz}{z} + \int_{\overline{\beta}}\frac{G'_-(\eta)}{G_-(\eta)}d\eta \\
			&\to \int_{\overline{\beta}}\frac{G'_-(\eta)}{G_-(\eta)}d\eta, \quad\epsilon\to0
			\end{split}		
			\end{equation}
			and 
			\begin{equation}
			\begin{split}
			\lim\limits_{\epsilon\to0}\int_{\partial D_{\epsilon}^2}\frac{\zeta'(s)}{\zeta(s)}ds &= 0.		
			\end{split}		
			\end{equation}
			
		\end{proof}		 
		
		The contours $\partial D_{\epsilon}^1$ and $\partial D_{\epsilon}^2$ can be chosen rather arbitrarily such that any domains in the critical strip can be covered by the domains $D_{\epsilon}^1$ or $D_{\epsilon}^2$. This leads to the final conclusion about the nontrivial zeros of $\zeta(s)$ when $s$ is far away from the real axis.  
		\begin{thm}\label{RH}
			Suppose $t$ is sufficiently large. 
			All the zeros of $\zeta(\sigma+it)$ are on the critical line $\{\sigma+it|\sigma=\frac{1}{2}\}$.
			And there exists one and only one zero of $\zeta(\frac{1}{2}+it)$ where $-2\pi (m+1)<\arg(\chi(\frac{1}{2}+it))<-2\pi m$. 	  
		\end{thm}
		\begin{proof}
			Let $s=\sigma+it$.
			
			The FAR domain $D$ in the $s$-plane can be separated into infinitely many horizontal strips   
			\begin{equation}
			D_m=D(\phi_2,\phi_1)=\bigcup_{\phi_2<\phi< \phi_1}\gamma_{\phi}
			\end{equation}
			where $m$ is sufficiently large integer and $\phi_2 = \phi_1 - 2\pi$, with their boundaries $\gamma_{\phi_1}$ and $\gamma_{\phi_2}$. 
			
			Step 1. Estimate the number of zeros and poles of $\zeta(s)$ in each domain $D_m$ with its boundary.
			
			Let $$\overline{D}_m=\bigcup_{\phi_2\le\phi\le\phi_1}\gamma_{\phi}$$
			and $D_0=\overline{D}_m\cap\{s|0<\sigma<1\}$ where $\phi_1=-2\pi m$. The domain $D_0$ is bounded and suppose $T=\sup\limits_{\sigma+it\in D_0}t$.
			
			By von Mangoldt's result \cite{Mangoldt05} the number of zeros in the domain $\{\sigma+it|0<\sigma<1,0<t\le T\}$ is $\frac{T}{2\pi}\log\frac{T}{2\pi}-\frac{T}{2\pi}+O(\log T)$, requiring a less number of zeros in each domain $\overline{D}_m$.
			The only pole of $\zeta(s)$ is at $s=1$.
			
			Therefore each domain $\overline{D}_m$ contains a finite number of zeros and no pole. 
			
			Step 2. Count the number of zeros of $\zeta(s)$ in each domain $D_m$.
			
			Suppose $\rho_1,\rho_2,\dots,\rho_n$ are the finite number of zeros of $\zeta(s)$ in the domain $D_m$. And the punctured domain $D_m\backslash\{\rho_1,\rho_2,\dots,\rho_n\}$ is multiply-connected.
			
			Supposing any $\rho_n\notin \{\sigma+it|\sigma=\frac{1}{2}\}$, we could find a domain symmetric with respect to the critical line in the connected $D_m$, with no zeros on its boundary except $\rho_n$ and $1-\overline{\rho}_n$. We could change the boundaries near $\rho_n$ and $1-\overline{\rho}_n$, constructing a pair of type-one simple domains, with one containing $\rho_n$ and the other not. The argument principle required that the logarithmic integrals along the boundaries of this pair of domains differ by $4\pi i$ for any small $\epsilon$. However, Lemma \ref{Res1} required that both integrals should tend to $2\pi i$ as $\epsilon\to0$. That made a contradiction.
			
			Therefore all the zeros of $\zeta(\sigma+it)$ in the domain $D_m$ are on the critical line, and Lemma \ref{Res1} rules the number to be one.
			In other words, each time $\chi(\frac{1}{2}+it)$ loops once around the origin from $\{1\}$ to $\{1\}$ on the unit circle $S^1$ with increasing $t$, there exists one zero of $\zeta(\frac{1}{2}+it)$.   
			
			Step 3. Count the number of zeros of $\zeta(s)$ on the boundary of domain $D_m$.			
			
			Let
			\begin{equation}
			D_m^2=D(\phi_2,\phi_1)=\bigcup_{\phi_2<\phi< \phi_1}\gamma_{\phi}
			\end{equation}
			where $\phi_1=\pi-2\pi m$ and $\phi_2=-\pi-2\pi m$.
			
			Supposing any $\rho\in\gamma_{\phi}$ where $\phi=-2\pi m$ such that $\zeta(\rho)=0$, then a type-two simple domain $D_{\epsilon}^2\subset D_m^2$ could be constructed to contain $\rho$. The argument principle required that the logarithmic integral along $\partial D_{\epsilon}^2$ to be no less than $2\pi i$ for any small $\epsilon$. However, Lemma \ref{Res2} required that the integral should tend to $0$ as $\epsilon\to0$. That made a contradiction.
			
			Therefore no zero of $\zeta(s)$ is on the boundary of domain $D_m$.
			In other words, $\zeta(s)\ne0$ when $\arg\chi(s)=-2\pi m$.
			
		\end{proof}
		
		Theorem \ref{RH} claims that the argument of $\chi(\frac{1}{2}+it)$ determines the distribution of zeros of $\zeta(s)$ on the critical line, which can be estimated by \eqref{argchi}. Roughly speaking, there exists one nontrivial zero when $t-t\log\frac{t}{2\pi}$ decreases by $2\pi$ for sufficiently large $t$. The result is more subtle than Riemann's initial hypothesis.		
		
		The distribution of the first 12 nontrivial zeros of $\zeta(\frac{1}{2}+it)$ is plotted in Figure \ref{myfigure2} as an illustration.
		The imaginary part $t$ serves as the Y axis, and the consecutive zeros appear rather in chaos along the Y axis. In fact it is claimed that there are infinitely many consecutive zeros with a gap at least 3 times larger than the average spacing, and there are also infinitely many consecutive zeros with a gap at most 0.6 times smaller than the average spacing.
		When $\frac{1}{2\pi}\arg(\chi(\frac{1}{2}+it))$ serves as the X axis, the consecutive zeros appear in order ruled by Theorem \ref{RH} along the X axis, though Theorem \ref{RH} applies for sufficiently large $t$. 
		
		\begin{figure}			
			\vspace{5pt}
			\scalebox{0.75}[0.75]{\includegraphics{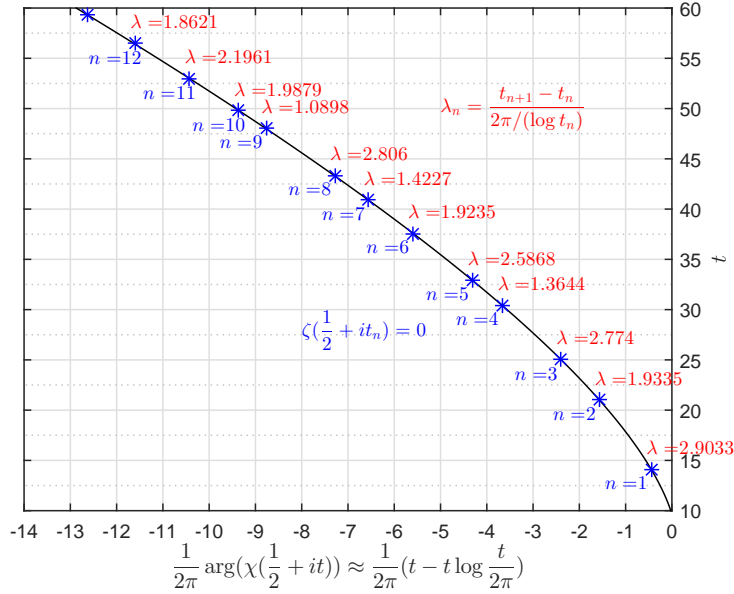}}
			\caption[]{Plot of the first 12 nontrivial zeros of $\zeta(\frac{1}{2}+it)$.}
			\label{myfigure2}
		\end{figure}
		
	\section{Conclusion}
		The chi function and tau function, as a pair of special functions similiar to the exponential function and logarithm function, can be expected to play more roles in the complex analysis than to explore the distribution of nontrivial zeros of Riemann zeta function. Riemann zeta function is shown to be a kind of branched functions which can be defined in a simple functional equation by a pair of conjugated branches. The proposed cross-branch technique is suitable to calculate the improper logarithmic integrals of this kind of functions along well-chosen contours symmetric with respect to the critical line.

	\bibliography{refer}
	\bibliographystyle{aomplain}
	
	\end{document}